\documentclass[a4paper,5pt]{article}
\usepackage[utf8]{inputenc}
\usepackage[english]{babel}
\usepackage{amsthm}
\usepackage{amsmath}
\usepackage{graphicx,latexsym}
\usepackage{adjustbox}
\usepackage{rotating}
\usepackage{eufrak}
\usepackage{amssymb}
\theoremstyle{definition}
\newtheorem{theorem}{Theorem}
\newtheorem{definition}{Definition}
\newtheorem{example}{Example}

\newtheorem{remark}{Remark}
\newtheorem{lemma}{Lemma}

\begin{document}

\begin{center}
\large\textbf{Rough bi-Heyting algebra and its applications in Heyting-Brouwer logic}\\~\\
\small\textbf{B. Praba$^{*}$, L.P.Anto Freeda}\\
\small\textit{Department of Mathematics, \\Sri Sivasubramaniya Nadar College of Engineering\\ Chennai, Tamil nadu, India.\\prabab@ssn.edu.in, antofreedalp@ssn.edu.in}\\

\rule{0mm}{6mm}

\renewcommand{\thefootnote}{}\footnotetext{\scriptsize
$^{\ast}$Corresponding author\\
\addtocounter{footnote}{-1} 
\footnotetext{\scriptsize This paper was submitted to arXiv under the reference number …………..}
{\it 2020 Mathematics Subject Classification}: 
\\
{\rm This paper was submitted to arXiv under the reference number: https://doi.org/10.48550/arXiv.2208.11851.}
}

\end{center}

\begin{abstract}
This paper focuses on defining a specialized distributive lattice known as rough bi-Heyting algebra on a rough semiring $(T,\Delta,\nabla)$, where $T$ is the set of all rough sets defined for the subsets of $U$ in a given approximation space $I=(U, R)$. In this study, the characterization of weaker complements namely pseudocomplement $(^{*})$, relative pseudocomplement $(\rightarrow)$, dual pseudocomplement $(^{+})$, and dual relative pseudocomplement $(\leftarrow)$ for the elements of $T$ are established. This characterization is instrumental for the existence and uniqueness of weaker complements on a rough semiring $(T,\Delta,\nabla)$.
The existence of relative pseudocomplement and dual relative pseudocomplement between the elements of $T$ supports to prove that $RBHey(T)=(T,\Delta,\nabla,^{*},^{+},\rightarrow,\leftarrow, RS(\emptyset), RS(U))$ is indeed a rough bi-Heyting algebra. In addition, these concepts are extended to demonstrate that $RBHey(T)$ is a rough stone algebra and dual rough stone algebra. Further, rough bi-Heyting algebra models various semantics of rough Heyting-Brouwer logic, amalgamating bi-Heyting algebra and Heyting-Brouwer logic on a rough semiring $(T,\Delta,\nabla)$.  The syntax for the rough Heyting-Brouwer logic is defined in terms of well-formed formulas incorporating the rough bi-Heyting algebraic operations. Finally, syntax validation is verified based on the rough Heyting-Brouwer logic's propositional and predicate calculus semantics.
\end{abstract}



\section{Introduction}
In recent times, rough set theory (RST) and its application have gained more attention among the researcher's group due to its nature of handling noisy and imprecise data. Zdzislaw I. Pawlak \cite{20}, a Polish computer scientist formalized this conventional set, called a rough set.  The idea of integrating RST with other theories is to enhance its ability to process complex datasets. The increased use of data is a significant advancement across various industries. This allows the researchers to refine their methodologies in the evolving field of RST to expand their applications in data analytics, artificial intelligence, data mining, machine learning, decision-making scenarios, etc. Numerous applications have raised attention to refine the conventional method. Among different adaptations to improve the traditional RST, one technique that stands out is substituting the clustering method for the equivalence relation. Data handling can be more flexible and sophisticated, allowing for a wider range of applications where similarity-based grouping would be appropriate. However, for the construction of classical algebraic structures on RST, non-overlapping equivalence classes of well-defined partitions are required. These partitions are crucial in the integrity of algebraic structures on RST.  
This paper establishes a new algebraic framework by taking rough semiring $(T, \Delta, \nabla)$ \cite{14} as an underlying structure. This is achieved by considering an approximation space $I = (U, R)$, where the universal set $U$ is the collection of a non-empty finite set of objects and $R$ is an arbitrary equivalence relation on $U$. The set of all rough sets is denoted by $T$ and each rough set is defined through the lower and upper approximation of subsets $X$ of $U$. The binary operations Praba $\Delta$ and Praba $\nabla$ \cite{13} defined between the subsets of $U$ are used to establish the existence of the least upper bound and greatest lower bound. Then it is proved that $T$ is a partially ordered set with its supremum and infimum obtained using Praba $\Delta$ and Praba $\nabla$ form a lattice called rough lattice \cite{13}. This rough lattice which is distributive with the operation Praba $\Delta$ and Praba $\nabla$ proves that $(T, \Delta, \nabla)$ is a semiring called a rough semiring \cite{14}. The detailed study of the structure of this semiring is studied \cite{10}. Building on this existing literature, many research works have been carried out that gave rise to various algebraic structures connected with distributive lattices. This paper is such a unique contribution to the field by defining bi-Heyting algebra within a rough semiring $(T, \Delta, \nabla)$ \cite{14}, presenting a new approach for modeling Heyting-Brouwer logic.\\

A mathematical structure involved in the study of non-classical logic is Heyting algebra. Heyting Algebra is named after Brouwer’s student, Arend Heyting in 1930. It is a distributive lattice that generalizes Boolean algebra. Heyting developed the methods and concepts used to solve the problems arising in intuitionistic logic. The key role of relative pseudocomplement in Heyting algebra is how well one element is relatively complemented with the other element. This is analogous to the implication operation in the context of intuitionistic logic.  Chajda and Langer \cite{9} study provides an extensive connection between relatively pseudocomplemented meet-semilattices and arbitrary posets to investigate the concept of relatively pseudocomplemented posets. Also, the authors establish the characterizations of relatively pseudocomplemented posets. But Ivan Chajda et al. \cite{8} have extended the relative pseudocomplement to posets (partially ordered sets) and characterized it using $L$-identities. In particular, they present the idea of congruences in this context. Adam Grabowski \cite{1} presents a formal introduction to pseudocomplement in a lattice. This work defines the concepts of the skeleton and the set of dense elements in a pseudocomplemented lattice. Though Adam Grabowski’s work gave the construction of stone algebra for any distributive lattice the work was first initiated by Raymond Balbes \cite{15}. Heyting algebra is extended to define double Heyting algebra on a lattice incorporating the properties of dual pseudocomplement and dual relative pseudocomplement was proposed by Sankappanavar \cite{16}. The algebraic approach to model the intuitionistic logic is found in the weaker complements of Heyting algebra. This contains the essence of intuitionistic reasoning between logic and lattice theory. The domain of Heyting-Brouwer propositional logic (H-B) serves as an interface between intuitionistic and classical logic. The H-B logic is an extension of intuitionistic propositional logic with the additional operators $(^{+})$ and $(\leftarrow)$. The article by Cecylia Rauszer \cite{4} established the features of the logical connectives unique to H-B logic. The link between H-B logic and other logical systems is also explored in this work. In the previous work \cite{5}, the author provides an in-depth understanding of the propositional calculus of H-B logic and makes a substantial contribution to the field of symbolic logic. Thomas Macaulay Ferguson \cite{19} establishes the addition of new connectives to the language through extensions of intuitionistic logic, such as Rauszer’s H-B logic (HB) and Nelson’s logic of constructible falsehood (N). The hierarchy of increasing semantics in intuitionistic logic ranging from the least Kripke semantic to more general algebraic semantics was given by Guram Bezhanishvili et al. \cite{7}. Arnold Beckmann and Norbert Preining \cite{3} examine the connection between Godel logic and intermediate predicate logic based on countable linear Kripke frames with constant domains. The study in \cite{2} builds on earlier research by Arnold Beckmann et al. \cite{3} to demonstrate that recursive enumeration is not possible for logic based on linear Kripke frames. Steffen Lewitzka \cite{17} examines the well-known outcome that states that a formula $\phi$ is an intuitionistic propositional logic (IPC) theorem if and only if its Godel-translation $\phi^{\prime}$ is a modal logic S4 theorem. He introduces a conservative extension of classical propositional logic (CPC), known as classical modal logic L, which prevents the collapse of both logics and has a copy of IPC. Piero Pagliani \cite{11} indicates the rough set features associated with rough bottom and rough top equality through logic-algebraic operations. Properties unique to rough equalities have been inferred by systematically examining the connections between the negation operations in the Nelson algebra. Also, there is an interesting analogy between semi-simple Nelson algebras and three-valued Lukasiewicz algebras by Piero Pagliani \cite{12}. Francesc Esteva et al. \cite{6} define distributive involutive residuated lattices. The author investigates the connection between para consistency and Nelson’s constructive logic, emphasizing how Nelson’s logic is an algebraic semantic subclass of residuated lattices. Stephen Comer \cite{18} explores the connection between algebraic logic, rough sets, and information systems, with a special emphasis on indiscernibility relations and approximation spaces. This formally proposes the set of axioms for rough relation algebra which makes it linkages to the study of ordinary relation algebra.\\

Heyting algebra allows us to reason constructively and investigate mathematical ideas that fall outside the classical logic, by serving as a bridge between algebraic structures and intuitionistic logic. This work's rough set theoretic ideas aim to define rough bi-Heyting algebra by considering a rough semiring $(T, \Delta, \nabla)$. Even though in the literature \cite{11,12,18} the weaker complements help in obtaining Heyting algebra is discussed for the rough sets in a bounded distributive lattice. In contrast, this work establishes the distinct characterization of weaker complements for the rough semiring using operations Praba $\Delta$ and Praba $\nabla$ which has not been so far addressed in the literature. For the considered approximation space, the arbitrary equivalence relation on $U$ induces disjoint equivalence classes which is significant to this work. This is crucial in preserving algebraic structures while providing characterization for the complements. For the existence of bi-Heyting algebra, the characterization of pseudocomplement and relative pseudocomplement for the elements of $T$ along with its dual are established. The characterization of pseudocomplement and its dual are also extended in defining the double rough stone algebra. Due to its exploration of pseudocomplement and relative pseudocomplement along with its dual, bi-Heyting algebra contains a hierarchical array of algebraic structures between Boolean algebra and lattices. 
Since analyzing rough sets through Heyting algebra is discussed for the approximation spaces and indiscernibility relations, this paper explores the association between these concepts through the H-B logic. The bi-intuitionistic logic which involves constructive proof theories and a semantic approach inspired by Heyting-Brouwer logic \cite{4,5,19} proposes models considering a partially ordered set. The relative pseudocomplement in Heyting algebra and the dual relative pseudocomplement in dual Heyting algebra are analogous to the implication and dual implication respectively which capture the idea of how a proposition implies the other in a bi-intuitionistic logic framework. The motivation to explore H-B logic in rough bi-Heyting algebra is the characterization of Heyting’s negation $(^{*})$ and Brouwer’s negation $(^{+})$. It is defined in such a way that the law of excluded middle that does not hold for the H-B logic in the classical case also does not hold for the rough Heyting-Brouwer logic. The construction of propositional and predicate calculus for the H-B logic validates the model in a partially ordered set of worlds. In this paper, for the given approximation space and a rough bi-Heyting algebra each rough set in $T$ is assumed to be world taken to validate the semantics. Each object in $U$ is a propositional variable and the possible conclusion will be the ordered pair of 0's and 1's obtained through the lower and upper approximation of the rough set. The well-formed formulas formed from the rough sets using the rough bi-Heyting algebraic connectives Praba $\Delta$, Praba $\nabla$, $\rightarrow$, $\leftarrow$, and unary operators $(^{*})$, $(^{+})$ also will yield the conclusion through the ordered pair of 0's and 1's.\\

The literature study provided in this work investigates the role of weaker complements for obtaining Heyting algebra in rough sets within a bounded distributive lattice. However, there is a lack of research on the unique characterization of weaker complements on a rough semiring $(T,\Delta,\nabla)$. Although the idea of analyzing rough sets through Heyting algebra has been discussed, the association between these concepts and H-B logic for validating the semantics by our proposed method has not been investigated, especially within the context of rough bi-Heyting algebra. This gap serves as a strong motivation behind this work to bridge the gap between algebraic structures and H-B logic by introducing rough bi-Heyting algebra, which is not yet addressed in the literature. Additionally, this paper seeks to explore H-B logic within rough bi-Heyting algebra, focusing on the characterization of Heyting’s negation ($^{*}$), Brouwer’s negation ($^{+}$), intuitionistic implication ($\rightarrow$), dual intuitionistic implication ($\leftarrow$) and validating semantics within the rough set framework. Unlike the traditional method, the proposed approach uses a rough set not only to determine the validity (or) non-validity of the semantic but also its possible satisfaction. This is a significant approach in the decision-making scenario different from the existing approach which does not account for the possible satisfaction of the semantics.\\

This paper unfolds into distinct sections. Preliminary definitions of rough semiring $(T,\Delta,\nabla)$ are provided in Section 2. The existence of rough bi-Heyting algebra on a rough semiring $(T,\Delta,\nabla)$ is demonstrated in Section 3. Establishing the semantics of rough H-B logic with rough bi-Heyting algebraic operators is discussed in Section 4. The paper's conclusion and the outline of future work are presented in Section 5.

\section{Preliminaries}
In this section,  the formal definitions of rough set are given.

\begin{definition}\cite{20}
 The structure $I=(U, R)$ is an approximation space where $U$ is a non-empty finite set of objects, called a universe and $R$ is an arbitrary equivalence relation on $U$. A partition induced by the relation $R$ consists of an equivalence class denoted $[x]_{R}$, which is a subset of $U$. The lower and upper approximation defined for the subset $X$ of $U$ is
\begin{center}
$R_{-}(X)=\{x \in U \mid [x]_{R} \subseteq X\}$\\  $R^{-}(X)=\{x \in U \mid [x]_{R} \cap X \neq \emptyset\}$ 
\end{center}   
 If $X$ is an arbitrary subset of $U$, then the rough set $RS(X)$ is an ordered pair $(R_{-}(X), R^{-}(X))$. The set of rough sets is denoted by $T$ and defined by $T=\{RS(X) \mid X \subseteq U\}$.
\end{definition}

\begin{definition}\cite{13}
 Let $X, Y \subseteq U$. The Praba join of $X$ and $Y$ is denoted by $X \Delta Y$ and defined as $X \Delta Y=X \cup Y$ if $IW(X \cup Y)=IW(X)+IW(Y)-IW(X \cap Y)$, where $IW(X)$ is the number of equivalence classes in $X$. \\ 
If $IW(X \cup Y) > IW(X)+IW(Y)-IW(X \cap Y)$, then the equivalence classes obtained by the union of $X$ and $Y$ are identified. The elements of that class belonging to $Y$ are deleted and the new set is named $Y$. Now, we obtain $X \Delta Y$. This process is repeated until 
\begin{center}
$IW(X \cup Y)=IW(X)+IW(Y)-IW(X \cap Y)$
\end{center}
\end{definition}

\begin{definition} \cite{13}
If $X,Y \subseteq U$, then an element $x \in U$ is called pivot element, if $[x]_{R} \nsubseteq X \cap Y$, but $[x]_{R} \cap X \neq \emptyset$ and $[x]_{R} \cap Y \neq \emptyset$.
\end{definition}

\begin{definition} \cite{13}
If $X, Y \subseteq U$, then the set of pivot elements of $X$ and $Y$ is called the pivot set of $X$ and $Y$ and is denoted by $P_{X \cap Y}$.
\end{definition}

\begin{definition} \cite{13}
Let $X, Y \subseteq U$. The Praba meet of $X$ and $Y$ is denoted as $X \nabla Y$ and defined by $X \nabla Y=\{x \in U \mid [x]_{R} \subseteq X \cap Y\} \cup P_{X \cap Y}$. Here each pivot element in $P_{X \cap Y}$ is the representative of that particular class.
\end{definition}

\begin{theorem}\cite{13}
For any two sets $X,Y$ in $U$,
\begin{enumerate}
	\item $RS(X \Delta Y)$ is the least upper bound of $RS(X)$ and $RS(Y)$
	\item $RS(X \nabla Y)$ is the greatest lower bound of $RS(X)$ and $RS(Y)$
\end{enumerate}
\end{theorem}
 
\begin{theorem}\cite{14}
For any given approximation space $I=(U, R)$, $(T,\Delta,\nabla)$ is a semiring called a rough semiring.
\end{theorem}

\section{Rough bi-Heyting algebra on $(T,\Delta,\nabla)$}
Throughout this section, we employ $I=(U, R)$ as an approximation space and $(T,\Delta,\nabla)$ is a rough semiring \cite{14}.
\begin{definition}
In a rough semiring $(T,\Delta,\nabla)$, an element $RS(X) \in T$ is said to have pseudocomplement if there exists a greatest element $RS(X)^{*} \in T$, disjoint from $RS(X)$ if $RS(X) \nabla RS(X)^{*}=RS(X \nabla X^{*})=RS(\emptyset)$. The pseudocomplement $RS(X)^{*}$ is defined by
 \begin{center}
 $RS(X)^{*}= \underset{RS(Y)}{\max} \{RS(Y) \in T \mid RS(X) \nabla RS(Y)=RS(X \nabla Y)= RS(\emptyset)\}$
 \end{center}
\end{definition}

\begin{definition}
A pseudocomplemented rough semiring $(T,\Delta,\nabla,^{*},RS(\emptyset),RS(U))$ is a rough semiring $(T,\Delta,\nabla)$ with the least element $RS(\emptyset)$ if every element in $T$ has a pseudocomplement.
\end{definition}
Note that the pseudocomplemented rough semiring $(T,\Delta,\nabla,^{*},RS(\emptyset),RS(U))$ has a greatest element (say) $RS(\emptyset)^{*}=RS(U)$, then $RS(\emptyset) \nabla RS(\emptyset)^{*}=RS(\emptyset) \nabla RS(U) =RS(\emptyset \nabla U)=RS(\emptyset)$.

\begin{theorem} 
The pseudocomplement of $RS(X)$ is unique.
\end{theorem}
\begin{proof}
For any $RS(X) \in T$, the pseudocomplement of $RS(X)$ is defined by
\begin{center}
$RS(X)^{*}= \underset{RS(Y)}{\max} \{RS(Y) \in T \mid RS(X)\nabla RS(Y)=RS(X \nabla Y)= RS(\emptyset) \}$ 
\end{center}
Now claim $RS(X)^{*}= RS(E-R^{-}(X))$ to prove $RS(X)\nabla RS(E-R^{-}(X))=RS(\emptyset)$. Let $E$ be the set of equivalence classes defined on $U$. Consider, $RS(X) \nabla RS(E-R^{-}(X)) = RS(X \nabla E-R^{-}(X))$. 
\begin{center}
 $X \nabla (E-R^{-}(X))= \{x \in U \mid [x]_{R} \subseteq X \cap E-R^{-}(X)\} \cup P_{X \cap E-R^{-}(X)} = A \cup B $ (say) 
\end{center}
  where $A=\{x \in U \mid [x]_{R} \subseteq X \cap E-R^{-}(X)\}$ and $B=P_{X \cap E-R^{-}(X)}$.\\
Since $X \cap  E-R^{-}(X)= \emptyset$ then $A=\emptyset$. Also, there is no pivot element common to $X$ and $E-R^{-}(X)$. Hence $B=\emptyset$.
	\begin{center}
	$\therefore RS(X \nabla E-R^{-}(X))=RS(\emptyset)$
	\end{center}
\textbf{To Prove} $RS(X)^{*}$ is maximum \\
Let $RS(X_{1})^{*}$ be the maximum such that $RS(X) \nabla RS(X_{1})^{*}=RS(X \nabla X_{1}^{*})= RS(\emptyset)$. To prove $RS(X)^{*}$ is maximum, it is enough to prove $RS(X)^{*}=RS(X_{1})^{*}$. It is clear that
\begin{equation}
 RS(X)^{*} \subseteq RS(X_{1})^{*} 
\end{equation}
Now to prove $RS(X_{1})^{*} \subseteq RS(X)^{*}$, it is enough to prove $(R_{-}(X_{1}^{*}), R^{-}(X_{1}^{*})) \subseteq (R_{-}(X^{*}), R^{-}(X^{*}))$. So claim $R_{-}(X_{1}^{*}) \subseteq R_{-}(X^{*})$. \\Let $x \in R_{-}(X_{1}^{*})$ then 
 $[x]_{R} \subseteq X_{1}^{*}$ but $[x]_{R} \cap X = \emptyset$
\begin{center}
$\Rightarrow  [x]_{R} \subseteq X^{*}$ (since $X_{1}^{*} \subseteq X^{*}$)\\
$\Rightarrow  x \in R_{-}(X^{*})$\\
$\therefore {R_{-}(X_{1}^{*}) \subseteq R_{-}(X^{*})}$
\end{center}
Now claim $R^{-}(X_{1}^{*}) \subseteq R^{-}(X^{*})$. Let $y \in R^{-}(X_{1}^{*})$ then $[y]_{R} \cap X_{1}^{*} \neq \emptyset$. If $[y]_{R} \cap X \neq \emptyset$ then $y \in R_{X \cap X_{1}^{*}}$ is not possible since $RS(X_{1}^{*} \nabla X)=RS(\emptyset)$.
\begin{center}
$\Rightarrow  [y]_{R} \cap X = \emptyset$ \\
$\Rightarrow  [y]_{R} \cap X^{*} \neq \emptyset$\\
$\Rightarrow  y \in R^{-}(X^{*})$ \\
$\therefore R^{-}(X_{1}^{*}) \subseteq R^{-}(X^{*})$\\
$\therefore {(R_{-}(X_{1}^{*}), R^{-}(X_{1}^{*})) \subseteq (R_{-}(X^{*}), R^{-}(X^{*}))}$
\end{center}
\begin{equation}
RS(X_{1})^{*} \subseteq RS(X)^{*}    
\end{equation}
Hence the pseudocomplement $RS(X)^{*}$ exists and it is a maximum. Therefore, the pseudocomplement of $RS(X)$ is unique.
\end{proof}

\begin{example}
Consider an approximation space $I=(U,R)$ where $U=\{x_{1},x_{2},x_{3},x_{4},x_{5},x_{6}\}$ be the universe and $R$ is an arbitrary equivalence relation on $U$. 

The equivalence classes obtained by the relation $R$ are given by
$X_{1}=[x_{1}]_{R}=\{x_{1},x_{3}\}$, $X_{2}=[x_{2}]_{R}=\{x_{2},x_{4},x_{6}\}$ and $X_{3}=[x_{5}]_{R}=\{x_{5}\}$. Then the collection of rough sets obtained are,
\begin{center}

$T=\{RS(\emptyset),RS(\{x_{1}\}),RS(\{x_{2}\}),RS(X_{1}),RS(X_{2}),RS(X_{3}),RS(\{x_{1}\} \cup \{x_{2}\}),RS(X_{1} \cup \{x_{2}\}),RS(\{x_{1}\} \cup X_{2}),RS(\{x_{1}\} \cup X_{3}),RS(\{x_{2}\} \cup X_{3}),RS(X_{1} \cup X_{3}),RS(X_{2} \cup X_{3}),RS(X_{1} \cup X_{2}),RS(\{x_{1}\} \cup \{x_{2}\} \cup X_{3}),RS(X_{1} \cup \{x_{2}\} \cup X_{3}),RS(\{x_{1}\} \cup X_{2} \cup X_{3}),RS(U)\}$.    
\end{center}
The pseudocomplement of elements in $T$ is given in Table 1,

\begin{table}[!ht]
\begin{adjustbox}{width=0.8\columnwidth,center}
\tiny
	\centering
		\begin{tabular}{|c|c|}
		\hline $RS(X)$ & $RS(X)^{*}$ \\ \hline $RS(\emptyset)$ & $RS(U)$\\ \hline  $RS(\{x_{1}\})$,$RS(X_{1})$ & $RS(X_{2} \cup X_{3})$\\ \hline $RS(\{x_{2}\})$,$RS(X_{2})$ & $RS(X_{1} \cup X_{3})$\\ \hline $RS(X_{3})$ & $RS(X_{1} \cup X_{2})$ \\ \hline $RS(\{x_{1}\} \cup \{x_{2}\})$,$RS(X_{1} \cup \{x_{2}\})$,$RS(\{x_{1}\} \cup X_{2})$,$RS(X_{1} \cup X_{2})$ & $RS(X_{3})$ \\ \hline $RS(\{x_{1}\} \cup X_{3})$,$RS(X_{1} \cup X_{3})$ & $RS(X_{2})$\\ \hline $RS(\{x_{2}\} \cup X_{3})$,$RS(X_{2} \cup X_{3})$ & $RS(X_{1})$\\ \hline $RS(\{x_{1}\} \cup \{x_{2}\} \cup X_{3})$,$RS(X_{1} \cup \{x_{2}\} \cup X_{3})$,$RS(\{x_{1}\} \cup X_{2} \cup X_{3})$,$RS(X_{1} \cup X_{2} \cup X_{3})$ & $RS(\emptyset)$\\ \hline
		\end{tabular}
		\end{adjustbox}
\caption{}
\label{Table}
\end{table}
\end{example}

\begin{example}
The following example indicates that $RS(X) \Delta RS(X)^{*}=RS(U)$ is not always true.\\

From Example 1, Consider $RS(\{x_{1}\} \cup \{x_{2}\}) \in T$ and its pseudocomplement is given by $RS(\{x_{1}\} \cup \{x_{2}\})^{*}=RS(E-R^{-}(\{x_{1}\} \cup \{x_{2}\}))=RS(E-X_{1} \cup X_{2})=RS(X_{3})$. Then \\

\begin{center}
$RS(\{x_{1}\} \cup \{x_{2}\}) \Delta RS(\{x_{1}\} \cup \{x_{2}\})^{*}=RS(\{x_{1}\} \cup \{x_{2}\}) \Delta RS(X_{3})=RS(\{x_{1}\} \cup \{x_{2}\} \cup X_{3}) \neq RS(U)$.    
\end{center}

Similarly for $RS(X_{2}) \in T$ and its pseudocomplement $RS(X_{2})^{*}=RS(E-R^{-}(X_{2}))=RS(X_{1} \cup X_{3})$ 

\begin{center}
$RS(X_{2}) \Delta RS(X_{2})^{*}=RS(X_{2}) \Delta RS(X_{1} \cup X_{3})=RS(U)$
\end{center}

\end{example}

\begin{remark}
When $X$ is the union of one (or) more equivalence classes then $RS(X) \Delta RS(X)^{*} =RS(U)$.
\end{remark}

\begin{definition}
A pseudocomplemented rough semiring $(T,\Delta,\nabla,^{*},RS(\emptyset),RS(U))$  is called a rough stone algebra if $RS(X)^{*} \Delta RS(X)^{**}=RS(U)$ for all $RS(X) \in T$.
\end{definition}

\begin{lemma}
The pseudocomplemented rough semiring $(T,\Delta,\nabla,^{*},RS(\emptyset),RS(U))$ is a rough stone algebra.
\end{lemma}
\begin{proof}
To prove $T$ is a rough stone algebra, it is enough to show that for any $RS(X) \in T$, 
\begin{equation}
RS(X)^{*} \Delta RS(X)^{**}=RS(U)  
\end{equation}
The pseudocomplement of any $RS(X) \in T$ is defined to be 
\begin{center}
$RS(X)^{*}=RS(E-R^{-}(X))$
\end{center}
and its
\begin{center}
 $RS(X)^{**}=RS(E-R^{-}(X))^{*}$\\
$=RS(E-R^{-}(E-R^{-}(X)))$ (since $RS(X)^{*}=RS(E-R^{-}(X))$)\\
$ =RS(E-(E-R^{-}(X)))=RS(R^{-}(X))$\\
\end{center}
Hence $RS(X)^{*} \Delta RS(X)^{**}=RS(E-R^{-}(X)) \Delta RS(R^{-}(X))=RS(E-R^{-}(X) \Delta R^{-}(X))=RS(E)=RS(U)$. Therefore (3) holds.
\end{proof}

\begin{example}
The following example shows that the rough stone algebra condition is satisfied by the elements of $T$.\\ 

From Example 1, 

\begin{center}
$RS(\{x_{1}\} \cup X_{3})^{*}=RS(E-R^{-}(\{x_{1}\} \cup X_{3}))=RS(E-X_{1} \cup X_{3})=RS(X_{2})$
\end{center}
and
\begin{center}
$RS(\{x_{1}\} \cup X_{3})^{**}=RS(R^{-}(\{x_{1}\} \cup X_{3}))=RS(X_{1} \cup X_{3})$
\end{center}
Therefore, $RS(\{x_{1}\} \cup X_{3})^{*} \Delta RS(\{x_{1}\} \cup X_{3})^{**}=RS(X_{2}) \Delta RS(X_{1} \cup X_{3})=RS(U)$. Thus, any element in $T$ satisfies the condition of rough stone algebra.  
\end{example}  

\begin{remark}
From Example 3, it is also clear that $RS(X)^{*} \nabla RS(X)^{**}=RS(\emptyset)$.
\end{remark}

\begin{definition}
For $RS(Y), RS(Z) \in T$, an element $RS(Y) \rightarrow RS(Z)$ is the relative pseudocomplement of $RS(Y)$ with respect to $RS(Z)$ if $\underset{RS(W)}{\max} \{RS(W) \in T \mid RS(Y) \nabla RS(W) \leq RS(Z)\}$  
\end{definition}

\begin{remark}
The pseudocomplement of $RS(X)$ is also represented in terms of a relative pseudocomplement. It is $RS(X)^{*}=RS(X) \rightarrow RS(\emptyset)$.
\end{remark}

\begin{theorem}
The relative pseudocomplement of $RS(Y)$ with respect to $RS(Z)$ is unique.
\end{theorem}
\begin{proof}
Define
\begin{center}
$RS(Y) \rightarrow RS(Z)=RS((Y \nabla Z) \Delta (E-R^{-}(Y-Y \nabla Z)))$
\end{center}
Then the condition of relative pseudocomplement of $RS(Y)$ with respect to $RS(Z)$ is $RS(Y) \nabla (RS(Y) \rightarrow RS(Z)) \leq RS(Z)$
\begin{center}
$\Leftrightarrow RS(Y) \nabla RS((Y \nabla Z) \Delta (E-R^{-}(Y-Y \nabla Z))) \leq RS(Z)$
\end{center}
\begin{equation}
 \Leftrightarrow RS(Y \nabla((Y \nabla Z) \Delta (E-R^{-}(Y-Y \nabla Z))) \leq RS(Z)    
\end{equation}
Now the relative pseudocomplement condition is verified by solving (4) in the following cases.\\
\textbf{Case 1}\\
 Assume $Y \nabla Z=\emptyset$. Then from (4) 
 \begin{center}
 $RS(\emptyset) \leq RS(Z)$\\
 \end{center}
\textbf{Case 2}\\
Assume  $Y \nabla Z \neq \emptyset$. Then from (4), consider $RS(Y \nabla ((Y \nabla Z) \Delta (E-R^{-}(Y-Y \nabla Z)))$
\begin{center}
$=RS(Y \nabla (Y \Delta E-R^{-}(Y-Y \nabla Z)) \nabla (Z \Delta E-R^{-}(Y-Y \nabla Z))$\\
$\leq RS(Y \nabla Z)$\\
$\leq RS(Z)$\\
$\therefore RS(Y \nabla((Y \nabla Z) \Delta (E-R^{-}(Y-Y \nabla Z))) \leq RS(Z)$    
\end{center}
Now to prove $RS(Y) \rightarrow RS(Z)$ is the maximum. Suppose there exists $RS(K) \in T$ satisfying (4) then
\begin{equation}
RS(K \nabla Y) \leq RS(Z)
\end{equation} 
It is enough to prove $RS(K) \leq RS(Y) \rightarrow RS(Z)$. Let $x \in K$ and if $x \in Y \nabla Z$ then $x \in (Y \nabla Z) \Delta (E-R^{-}(Y-Y \nabla Z))$. Therefore, $K \subseteq (Y \nabla Z) \Delta (E-R^{-}(Y-Y \nabla Z))$
\begin{center}
$RS(K) \leq RS(Y) \rightarrow (Z)$
\end{center}
On the other hand, let $x \in K$ and $x \notin Y \nabla Z$. Suppose if $x \in Y$, then $x \in K \nabla Y$
\begin{center}
$\Rightarrow x \in Z$  (by (5))\\
$\Rightarrow x \in Y \nabla Z$
\end{center}
which is a contradiction.\\
Therefore, $x \notin Y \nabla Z$
\begin{center}
$\Rightarrow x \notin Y$\\
$\Rightarrow x \notin R^{-}(Y-Y \nabla Z)$\\
$\Rightarrow x \in E-R^{-}(Y-Y \nabla Z)$\\
$\Rightarrow x \in (Y \nabla Z) \Delta (E-R^{-}(Y-Y \nabla Z))$
\end{center}
Therefore, $K \subseteq (Y \nabla Z) \Delta (E-R^{-}(Y-Y \nabla Z))$ and 
\begin{center}
$RS(K) \leq RS(Y) \rightarrow RS(Z)$
\end{center}
Thus $RS(Y) \rightarrow RS(Z)$ is a maximum element in $T$ satisfying the condition of relative pseudocomplement. Hence the pseudocomplement of $RS(Y)$ relative to $RS(Z)$ exists and it is unique.
\end{proof}

\begin{definition}
A rough Heyting algebra (or) rough pseudo-boolean algebra $(T,\Delta,\nabla,^{*}, \rightarrow, RS(\emptyset), RS(U))$ is a rough semiring $(T,\Delta,\nabla)$ with the least element $RS(\emptyset)$ in which the relative pseudocomplement is defined for every pair of elements in $T$.
\end{definition}

\begin{theorem}
For any given approximation space $I=(U, R)$, the rough semiring $(T,\Delta,\nabla)$ is a rough Heyting algebra.
\end{theorem}
\begin{proof}
$(T,\Delta,\nabla)$ be a rough semiring \cite{14}, and by Theorem 4, the relative pseudocomplement exists for every pair of elements in $T$. Therefore, $(T,\Delta,\nabla,^{*}, \rightarrow,RS(\emptyset),RS(U))$ is a Rough Heyting algebra.
\end{proof}

\begin{definition}
Let $(T,\Delta,\nabla)$ be a rough semiring. A dual pseudocomplement of an element $RS(X)$ in $T$ is the least element $RS(X)^{+}$ if $RS(X) \Delta RS(X)^{+}=RS(U)$ and $RS(X)^{+}$ is defined as 
\begin{center}
$RS(X)^{+}= \underset{RS(Y)}{min} \{RS(Y) \in T \mid RS(X)\Delta RS(Y)=RS(X \Delta Y)= RS(U)\}$
\end{center}
\end{definition}

\begin{theorem}
The dual pseudocomplement of $RS(X)$ is unique.
\end{theorem}
\begin{proof}
For every $RS(X) \in T$, the dual pseudocomplement of $RS(X)^{+}$ is defined as,
\begin{center}
 $RS(X)^{+}= \underset{RS(Y)}{min} \{RS(Y) \in T/ RS(X) \Delta RS(Y)=RS(X \Delta Y)= RS(U) \}$
\end{center}
Let $E$ be the set of equivalence classes defined on $U$. Now claim $RS(X)^{+}=RS(E-R_{-}(X))$, to prove $RS(X) \Delta RS(X)^{+}=RS(U)$. Consider, 
\begin{center}
$RS(X) \Delta RS(E-R_{-}(X))=RS(X \Delta E-R_{-}(X))$
\end{center}
 Here,
\begin{center}
$X \Delta E-R_{-}(X)=X \cup E-R_{-}(X)$ if $IW(X \cup E-R_{-}(X))=IW(X)+IW(E-R_{-}(X))-IW(X \cap E-R_{-}(X))$\\
$\therefore  RS(X \Delta E-R_{-}(X))=RS(X \cup E-R_{-}(X))=RS(U)$
\end{center}
\textbf{To Prove} $RS(X)^{+}$ is minimum\\ Assume there is another minimal dual pseudocomplement $RS(X_{1})^{+}$ such that $RS(X) \Delta RS(X_{1})^{+}=RS(X \Delta X_{1}^{+})=RS(U)$. To prove $RS(X)^{+}$ is minimum, it is enough to show that $RS(X)^{+}=RS(X_{1})^{+}$. It is clear from the assumption that
\begin{equation}
RS(X_{1})^{+} \subseteq RS(X)^{+}
\end{equation}
Now to prove,  $RS(X)^{+} \subseteq RS(X_{1})^{+}$ it is enough to prove $(R_{-}(X^{+}),R^{-}(X^{+})) \subseteq (R_{-}(X_{1}^{+}),R^{-}(X_{1}^{+}))$\\ First claim $R_{-}(X^{+}) \subseteq R_{-}(X_{1}^{+})$\\ Let $x \in R_{-}(X^{+})$ then $[x]_{R} \subseteq X^{+}$
\begin{center}
$\Rightarrow [x]_{R} \subseteq X_{1}^{+}$ (since $X^{+} \subseteq X_{1}^{+}$)\\ 
$\Rightarrow x \in R_{-}(X_{1}^{+})$\\
$\therefore R_{-}(X^{+}) \subseteq R_{-}(X_{1}^{+})$
\end{center}
Now claim $R^{-}(X^{+}) \subseteq R^{-}(X_{1}^{+})$\\ Let $y \in R^{-}(X^{+})$ then $[y]_{R} \cap X^{+} \neq \emptyset$ (since $RS(X \Delta X^{+})=RS(U)$). If $[y]_{R} \cap X = \emptyset$ then $[y]_{R} \cap X_{1}^{+} \neq \emptyset$ (since $RS(X \Delta X^{+})=RS(U)$)
\begin{center}
 $\Rightarrow y \in R^{-}(X_{1}^{+})$\\
$\therefore R^{-}(X^{+}) \subseteq R^{-}(X_{1}^{+})$
\end{center}
If $[y]_{R} \cap X \neq \emptyset$ then $y \in R_{X \cap X^{+}}$ (since $RS(X \Delta X^{+})=RS(U)$)
\begin{center}
 $\Rightarrow [y]_{R} \cap  X_{1}^{+} \neq \emptyset$\\
 $\Rightarrow y \in R^{-}(X_{1}^{+})$ \\
$\therefore R^{-}(X^{+}) \subseteq R^{-}(X_{1}^{+})$
\end{center}
\begin{equation}
RS(X)^{+} \subseteq RS(X_{1})^{+}
\end{equation} 
From (6) and (7), it is obtained that $RS(X)^{+}=RS(X_{1})^{+}$. Therefore $RS(X)^{+}$ is the minimum and it exists for any $RS(X) \in T$. Hence the dual pseudocomplement of $RS(X)$ is unique. 
\end{proof}

\begin{example}
From Example 1, the dual pseudocomplement of the elements in $T$ is given in Table 2

\begin{table}[!ht]
\begin{adjustbox}{width=0.5\columnwidth,center}
	\centering
		\begin{tabular}{|c|c|}
\hline $RS(X)$ & $RS(X)^{+}$ \\ \hline $RS(\emptyset)$,$RS(\{x_{1}\})$,$RS(\{x_{2}\})$,$RS(\{x_{1}\} \cup \{x_{2}\})$ & $RS(U)$\\ \hline  $RS(X_{1})$,$RS(X_{1} \cup \{x_{2}\})$ & $RS(X_{2} \cup X_{3})$\\ \hline $RS(X_{2})$,$RS(\{x_{1}\} \cup X_{2})$ & $RS(X_{1} \cup X_{3})$\\ \hline $RS(X_{3})$,$RS(\{x_{1}\} \cup X_{3})$,$RS(\{x_{2}\} \cup X_{3})$,$RS(\{x_{1}\} \cup \{x_{2}\} \cup X_{3})$ & $RS(X_{1} \cup X_{2})$ \\ \hline $RS(X_{1} \cup X_{2})$ & $RS(X_{3})$ \\ \hline $RS(X_{1} \cup X_{3})$,$RS(X_{1} \cup \{x_{2}\} \cup X_{3})$ & $RS(X_{2})$\\ \hline $RS(X_{2} \cup X_{3})$,$RS(\{x_{1}\} \cup X_{2} \cup X_{3})$ & $RS(X_{1})$\\ \hline $RS(X_{1} \cup X_{2} \cup X_{3})$ & $RS(\emptyset)$\\ \hline
\end{tabular}
\end{adjustbox}
\caption{}
\label{Table}
\end{table}
\end{example}

\begin{remark}
The following example shows that $RS(X) \nabla RS(X)^{+}=RS(\emptyset)$ is not always true.    
\end{remark}

\begin{example}
From Example 1, Consider $RS(\{x_{1}\} \cup \{x_{2}\} \cup X_{3}) \in T$ whose dual pseudocomplement is given by $RS(\{x_{1}\} \cup \{x_{2}\} \cup X_{3})^{+}=RS(E-X_{3})=RS(X_{1} \cup X_{2})$. So that
\begin{center}
$RS(\{x_{1}\} \cup \{x_{2}\} \cup X_{3}) \nabla RS(\{x_{1}\} \cup \{x_{2}\} \cup X_{3})^{+}=RS(\{x_{1}\} \cup \{x_{2}\} \cup X_{3}) \nabla RS(X_{1} \cup X_{2})=RS(\{x_{1}\} \cup \{x_{2}\})$   
\end{center}
Similarly for $RS(X_{2} \cup X_{3}) \in T$ and its dual pseudocomplement $RS(X_{2} \cup X_{3})^{+}=RS(E-X_{2} \cup X_{3}))=RS(X_{1})$ 
\begin{center}
$RS(X_{2} \cup X_{3}) \nabla RS(X_{2} \cup X_{3})^{+}=RS(X_{2} \cup X_{3}) \nabla RS(X_{1})=RS(\emptyset)$
\end{center}
But this condition is verified when $X$ is the union of one (or) more equivalence classes.
\end{example}

\begin{definition}
A dual rough stone algebra is a dual pseudocomplemented rough semiring  $(T,\Delta,\nabla,^{+}, RS(\emptyset), RS(U))$ satisfying $RS(X)^{+} \nabla RS(X)^{++}=RS(\emptyset)$ for all $RS(X) \in T$.
\end{definition}

\begin{lemma}
The dual pseudocomplemented rough semiring $(T,\Delta,\nabla,^{+},RS(\emptyset),RS(U))$ is a dual rough stone algebra.
\end{lemma}
\begin{proof}
The proof is straightforward.   
\end{proof}

\begin{example} 
The following example describes the elements of $T$, satisfying the dual rough stone algebra condition.\\ 
From Example 1, 
\begin{center}
$RS(\{x_{1}\} \cup X_{2} \cup X_{3})^{+}=RS(E-X_{2} \cup X_{3})=RS(X_{1})$     
\end{center}
and 
\begin{center}
$RS(\{x_{1}\} \cup X_{2} \cup X_{3})^{++}=RS(X_{2} \cup X_{3})$  
\end{center}
 Then 
 \begin{center}
$RS(\{x_{1}\} \cup X_{2} \cup X_{3})^{+} \nabla RS(\{x_{1}\} \cup X_{2} \cup X_{3})^{++}=RS(X_{1}) \nabla RS(X_{2} \cup X_{3})=RS(\emptyset)$    
 \end{center} 
 Therefore, the dual rough stone algebra condition is satisfied by the elements of $T$.
\end{example}

\begin{definition}
A double rough stone algebra $(T,\Delta,\nabla,^{*},^{+}, RS(\emptyset), RS(U))$ is a rough stone algebra and dual rough stone algebra with the unary operation of pseudocomplementation and dual pseudocomplementation.
\end{definition}

\begin{definition}
For $RS(Y), RS(Z) \in T$, an element $RS(Y) \leftarrow RS(Z)$ is the dual relative pseudocomplement of $RS(Y)$ with respect to $RS(Z)$ if $\underset{RS(V)}{min} \{RS(V) \in T \mid RS(Z) \Delta RS(V) \geq RS(Y)\}$
\end{definition}

\begin{remark}
A dual pseudocomplement of $RS(X) \in T$ is expressed in terms of dual relative pseudocomplement and it is defined by $RS(X)^{+}=RS(U) \leftarrow RS(X)$.
\end{remark}

\begin{theorem}
The dual relative pseudocomplement of $RS(Y)$ with respect to $RS(Z)$ is unique.
\end{theorem}
\begin{proof}
Define
\begin{center}
$RS(Y) \leftarrow RS(Z)=RS(Y \nabla R^{-}(Y-Y \nabla Z))$
\end{center}
satisfy the condition that $RS(Z) \Delta (RS(Y) \leftarrow RS(Z)) \geq RS(Y)$
\begin{equation}
RS(Z) \Delta RS(Y \nabla R^{-}(Y-Y \nabla Z)) \geq RS(Y)
\end{equation}
\textbf{Case 1}\\
Assume $Y \nabla Z=\emptyset$. Then (8) becomes
\begin{center}
$RS(Z) \Delta RS(Y \nabla R^{-}(Y))=RS(Z \Delta (Y \nabla R^{-}(Y)))$\\
$=RS((Z \Delta Y) \nabla (Z \Delta R^{-}(Y))) \geq RS(Y)$\\
$\therefore RS(Z) \Delta (RS(Y) \leftarrow RS(Z)) \geq RS(Y)$
\end{center}
\textbf{Case 2}\\
Assume $Y \nabla Z \neq \emptyset$. Then from (8) 
\begin{center}
$RS(Z) \Delta RS(Y \nabla R^{-}(Y-Y \nabla Z))=RS((Z \Delta Y) \nabla (Z \Delta R^{-}(Y-Y \nabla Z))$
\end{center}
It is clear that $RS(Z \Delta Y) \geq RS(Y)$. Now to prove 
\begin{center}
$RS(Z \Delta R^{-}(Y-Y \nabla Z)) \geq RS(Y \Delta Z)$
\end{center}
where $RS(Z) \leq RS(Z \Delta R^{-}(Y-Y \nabla Z))$ and it remains to show that $RS(Y) \leq RS(Z \Delta R^{-}(Y-Y \nabla Z))$. Let $x \in Y$ then either $x \notin Y \nabla Z$ (or) $x \in Y \nabla Z$.\\
If $x \notin Y \nabla Z$ then $x \in Y - Y \nabla Z$
\begin{center}
$\Rightarrow x \in R^{-}(Y-Y \nabla Z)$\\
$\Rightarrow x \in Z \Delta R^{-}(Y-Y \nabla Z)$\\
$\therefore Y \subseteq Z \Delta R^{-}(Y-Y \nabla Z)$\\
$\Rightarrow RS(Y) \leq RS(Z \Delta R^{-}(Y-Y \nabla Z))$
\end{center}
If $x \in Y \nabla Z$ then $x \notin Y - Y \nabla Z$
\begin{center}
$\Rightarrow x \notin R^{-}(Y-Y \nabla Z)$\\
$\Rightarrow x \in Z \Delta R^{-}(Y-Y \nabla Z)$  (since $x \in Y \nabla Z)$\\
$\therefore Y \subseteq Z \Delta R^{-}(Y-Y \nabla Z)$\\
$\Rightarrow RS(Y) \leq RS(Z \Delta R^{-}(Y-Y \nabla Z))$
\end{center}
Therefore by the assumptions
\begin{center}
$RS(Z \Delta R^{-}(Y-Y \nabla Z)) \geq RS(Y \Delta Z)$
\end{center}
Hence (8) becomes 
\begin{center}
$RS(Z) \Delta RS(Y \nabla R^{-}(Y-Y \nabla Z))=RS((Z \Delta Y) \nabla (Z \Delta R^{-}(Y-Y \nabla Z))$\\
$\geq RS(Z \Delta Y) \geq RS(Y)$\\
$\therefore RS(Z) \Delta (RS(Y) \leftarrow RS(Z)) \geq RS(Y)$
\end{center}
Now to prove $RS(Y) \leftarrow RS(Z)$ is the minimum. If there exists $RS(K) \in T$ satisfying (8) then 
\begin{equation}
RS(K \Delta Z) \geq RS(Y)
\end{equation}
Now it is enough to prove $RS(Y) \leftarrow RS(Z) \leq RS(K)$. Let $x \in Y \nabla R^{-}(Y-Y \nabla Z)$
\begin{center}
$\Rightarrow x \in Y$ and $ x \notin Z$\\
$\Rightarrow x \in K \Delta Z$  (by (9))
\end{center}
But $x \notin Z$
\begin{center}
$\Rightarrow x \in K$\\
$Y \nabla R^{-}(Y-Y \nabla Z) \subseteq K$\\
$RS(Y) \leftarrow RS(Z) \leq RS(K)$
\end{center}
Therefore, $RS(Y) \leftarrow RS(Z)$ is the minimum element in $T$ for the dual relative pseudocomplement of $RS(Y)$  with respect to $RS(Z)$. Hence the minimum exists and it is unique.
\end{proof}

\begin{definition}
A dual rough Heyting algebra $(T,\Delta,\nabla,^{+}, \leftarrow, RS(\emptyset), RS(U))$ (or) dual rough pseudo-boolean algebra is rough semiring $(T,\Delta,\nabla)$ with the greatest element $RS(U)$ in which the dual relative pseudocomplement is defined for every pair of elements in $T$. 
\end{definition}

\begin{theorem}
For the given approximation space $I=(U, R)$, the rough semiring $(T,\Delta,\nabla)$ is a dual rough Heyting algebra.
\end{theorem}
\begin{proof}
$(T,\Delta,\nabla)$ be a rough semiring \cite{14}, and from Theorem 7, the dual relative pseudocomplement exists for every pair of elements in $T$. Hence $(T,\Delta,\nabla,^{+}, \leftarrow,RS(\emptyset),RS(U))$ is a dual rough Heyting algebra.
\end{proof}

\begin{definition}
A rough bi-Heyting algebra $(T, \Delta, \nabla,^{*}, ^{+},  \rightarrow, \leftarrow, RS(\emptyset), RS(U))$ is a rough semiring $(T,\Delta,\nabla)$ with a rough Heyting algebra and a dual rough Heyting algebra.
\end{definition}

\begin{theorem}
For the given approximation space $I=(U,R)$ and a rough semiring $(T,\Delta,\nabla)$, $RBHey(T)=(T, \Delta, \nabla,^{*}, ^{+},  \rightarrow, \leftarrow, RS(\emptyset), RS(U))$ is a rough bi-Heyting algebra.
\end{theorem}
\begin{proof}
The proof is straightforward.
\end{proof}

\begin{remark}
Let $(T, \Delta, \nabla,^{*}, ^{+},  \rightarrow, \leftarrow, RS(\emptyset), RS(U))$ be a rough bi-Heyting algebra. Then the following properties  
\begin{enumerate}
	\item $RS(X) \rightarrow RS(Y)=RS(U)$ iff $RS(X) \leq RS(Y)$
	\item $RS(X) \rightarrow RS(U)=RS(U)$
	\item $RS(\emptyset) \rightarrow RS(X)=RS(U)$
	\item $RS(U) \rightarrow RS(X)=RS(X)$
	\item $RS(X) \leftarrow RS(Y)=RS(\emptyset)$ iff $RS(X) \leq RS(Y)$
	\item $RS(X) \leftarrow RS(U)=RS(\emptyset)$
	\item $RS(X) \leftarrow RS(\emptyset)=RS(X)$
	\item $RS(\emptyset) \leftarrow RS(X)=RS(\emptyset)$
\end{enumerate}
holds for all $RS(X),RS(Y) \in T$.
\end{remark}

\section{Heyting-Brouwer logic in the context of Rough bi-Heyting algebra}

This section introduces Brouwer’s inspired study of intuitionistic logic through rough bi-Heyting algebraic axiomatic calculus. As bi-intuitionistic logic introduced \cite{5} is the conventional extension of intuitionistic propositional calculus in the study of bi-Heyting algebras. This section explores the rough H-B logic through propositional and predicate calculus semantics. In the existing literature, bi-Heyting algebras were employed to model bi-intuitionistic logic. Instead of emphasizing truth-valuation, bi-intuitionistic logic focuses more on justification based on evidence and probability. Arend Heyting had previously established intuitionistic logic to formalize Brouwer’s intuitionism, where the intuitionistic logic ideas are derived from classical logic. The difference in their interpretation is notably the law of excluded middle, which is permissible in classical logic and not applicable in intuitionistic logic. Consequently, the law of excluded middle for the weaker notions of complement $(^{*})$ and $(^{+})$ in rough bi-Heyting algebra for the operations Praba $\Delta$ and Praba $\nabla$ respectively does not applicable for all $RS(X)$ in $T$. This is a common feature where the rough bi-Heyting algebra can be considered to model the rough Heyting-Brouwer logic.\\

Consider an approximation space $I=(U, R)$, the equivalence classes $\{X_{1}, X_{2}, ...X_{n}\}$ are obtained using $R$ and the cardinality of $m$ equivalence classes $\{X_{1}, X_{2}, ...X_{m}\}$ are more than one. Then the cardinality of remaining $n-m$ equivalence classes $\{X_{m+1}, X_{m+2}, ...X_{n}\}$ is equal to one. For each subset $X$ of $U$, the rough set $RS(X)$ is defined through the lower and upper approximation of $X$. Let $T$ denote the set of all rough sets and in Section 3, we prove that $(T,\Delta,\nabla,^{*},^{+},\rightarrow,\leftarrow, RS(\emptyset), RS(U))$ is a rough bi-Heyting algebra. This algebraic structure is taken to model Heyting-Brouwer's logic by imposing the logical connection into an approximation space. Suppose that each object in $U$ symbolizes a logical statement (or) a propositional variable taking the value true (or) false, indicated by 1 (or) 0 respectively. 

\subsection{Syntax}
The language of rough Heyting-Brouwer logic is defined using the following
\begin{enumerate}
    \item A rough set $RS(\{x\})$: $x \in U$ is a propositional variable in the universe
    \item connectives: Binary connectives $\Delta, \nabla, \rightarrow, \leftarrow$ and unary connectives $^{*}, ^{+}$ (Ref: Section 3)
    \item Truth values: (0,0), (1,1), and (0,1) mean the non-satisfiability, satisfiability, and possible satisfiability
    \item Parentheses: (,) 
\end{enumerate}

\begin{definition}
The well-formed formula for rough Heyting-Brouwer logic is defined by
\begin{enumerate}
    \item $\{RS(\{x\}) \mid x \in U\}$ is a well-formed formula, where $x$ is a propositional variable
    \item $RS(\varPhi)$ is a well-formed formula, where $RS(\varPhi)$ is obtained from the connectives and parentheses
    \item If $RS(\varPhi), RS(\varPsi)$ are well-formed formula, then $RS(\varPhi) \Delta RS(\varPsi), RS(\varPhi) \nabla RS(\varPsi),RS(\varPhi)^{*}, RS(\varPsi)^{*},RS(\varPhi)^{+},$\\$ RS(\varPsi)^{+}, RS(\varPhi) \rightarrow RS(\varPsi), RS(\varPhi) \leftarrow RS(\varPsi)$ are all well-formed formulas
\end{enumerate}
\end{definition}

\begin{definition}
The logical axioms of rough Heyting-Brouwer logic are 
\begin{enumerate}
\item $RS(\emptyset) \rightarrow RS(X)$
\item $RS(U) \leftarrow RS(\emptyset)$
\item $RS(X) \rightarrow RS(X)$
\item $RS(X) \rightarrow (RS(X) \leftarrow RS(\emptyset))$
\item $RS(X) \rightarrow (RS(X) \Delta RS(Y))$
\item $RS(Y) \rightarrow (RS(X) \Delta RS(Y))$
\item $(RS(X) \nabla RS(Y)) \rightarrow RS(X)$
\item $(RS(X) \nabla RS(Y)) \rightarrow RS(Y)$
\item $(RS(X) \nabla RS(X)^{*}) \rightarrow RS(Y)$
\item $RS(X) \rightarrow (RS(Y) \Delta RS(Y)^{+})$
\end{enumerate}
\end{definition}

\begin{definition} 
For the rough bi-Heyting algebra $RBHey(T)=(T,\Delta,\nabla,^{*},^{+},\rightarrow,\leftarrow, RS(\emptyset), RS(U))$, the truth value of a propositional variable is defined through the function $TrV: U \rightarrow \{0,1\}$. For the subsets $X_{i} \subseteq U$ for $i=1,2, \ldots n$, the truth value of $X_{i}=\{x_{i_{1}}, x_{i_{2}},\ldots x_{i_{t}} \}$ is defined by,
\begin{center}
$TrV(X_{i})= \underset{x_{i_{j}} \in X_{i}}{min}  TrV(x_{i_{j}})$, where $j=1,2,...t$
\end{center}
where $\{x_{i_{1}}, x_{i_{2}},\ldots x_{i_{t}}\}$ are the propositional variables of the subset $X_{i}$ of $U$, taking the truth value 0 (or) 1. 
\end{definition} 

\begin{definition}
For any $RS(X) \in T$, the truth value of $RS(X)$ is defined by
\begin{center}
$R-TrV(RS(X))=(R-TrV(R_{-}(X)), R-TrV(R^{-}(X)))$
\end{center}
where $R-TrV(R_{-}(X))=\max \{TrV(X_{k}) \mid X_{k} \in R_{-}(X)\}$ and $R-TrV(R^{-}(X))=\max \{TrV(X_{l}) \mid X_{l} \in R^{-}(X)\}$
\end{definition} 
Hence we have the rough truth value function defined by $R-TrV: T \rightarrow \{(0,0),(0,1),(1,1)\}$
\begin{remark}
For the equivalence classes in $U$, if any of the propositional variables in an equivalence class has the truth value 0, then the truth value of an equivalence class is 0, otherwise 1. If any one of the equivalence classes in $R_{-}(X)$ (or) $R^{-}(X)$ has the truth value 1, then the rough truth value of $R_{-}(X)$ (or) $R^{-}(X)$ is 1, otherwise 0. 
\end{remark}
Note that (1,0) is not included in the co-domain as it is not a permissible truth value for any $RS(X)=(R_{-}(X), R^{-}(X))$

\begin{definition} 
A model $\mathbb{M}$ is a structure $\mathbb{M}=(RBHey(T), R-TrV)$, where $RBHey(T)$ is a rough bi-Heyting algebra and $R-TrV$ is a rough truth value function. The satisfiability of the well-formed formula (wff) using the function $R-TrV$ is given by

\begin{enumerate}
    \item A wff $RS(\varPhi)$ is said to be satisfiable if $R-TrV(RS(\varPhi))=(1,1)$
    \item A wff $RS(\varPhi)$ is said to be possibly satisfiable if $R-TrV(RS(\varPhi))=(0,1)$
    \item A wff $RS(\varPhi)$ is said to be not satisfiable if $R-TrV(RS(\varPhi))=(0,0)$
\end{enumerate}
\end{definition}

\begin{definition}
Given a wff $RS(\varPhi)$ under a model $\mathbb{M}$, $RS(\varPhi)$ is said to hold (or) satisfiable in the model is denoted by $\mathbb{M} \vDash^{S} RS(\varPhi)$ iff $R-TrV(RS(\varPhi))=(1,1)$ and not satisfiable in the model $\mathbb{M}$ is denoted by $\mathbb{M} \nvDash^{N} RS(\varPhi)$ iff $R-TrV(RS(\varPhi))=(0,0)$. Also the possible satisfiability of wff $RS(\varPhi)$ in the model $\mathbb{M}$ is denoted by $\mathbb{M} \vDash^{P} RS(\varPhi)$ iff $R-TrV(RS(\varPhi))=(0,1)$ 

\end{definition}

The above discussion leads to defining the rough Heyting-Brouwer logic

\begin{definition}
Rough Heyting-Brouwer logic is the union of intuitionistic and dual intuitionistic logic, and its conservative extension is applied to the elements of rough bi-Heyting algebra. 
In other words, the intuitionistic and dual intuitionistic implications are defined for the rough sets.
\end{definition}

\subsection{Semantics for rough Heyting-Brouwer propositional calculus}
The Heyting-Brouwer (H-B) interpretation is the conventional explanation for bi-intuitionistic logic. If there is direct evidence, or constructive proof, supporting a proposition, it is considered true. Rather than valuing truth, bi-intuitionistic logic operations retain justification in evidence and probability. The most predominant approach in rough H-B propositional logic is the semantic approach. Here the propositional variables in $U$ are assigned a value, either true if 1 (or) false if 0. The value of the proposition is not always universally absolute but varies for different worlds in $T$. Rough bi-Heyting algebraic elements are considered as worlds when formulating the semantics. This method of considering the world is extended under a rough set framework in the following way. 
Suppose if $RS(Z)$ is a formula, then the value of propositional variables in $RS(Z)$ is determined with respect to the world $RS(X)$ based on the presence of elements in the lower and upper approximation of $RS(X)$.
This permits us to calculate the rough truth value of any mathematical expression $RS(\varPhi)$ with respect to the world $RS(X)$ under the model considered.  An algebraic model with a satisfaction relation of these mathematical expressions using the operations of rough bi-Heyting algebra is rough Heyting-Brouwer propositional calculus (RBi-IPC). Now the formal semantics for the rough H-B propositional calculus will be defined.\\

For the model $\mathbb{M}=(RBHey(T), R-TrV)$ and the world $RS(X) \in T$, the satisfaction relation for the rough H-B propositional calculus is as follows

\begin{enumerate}
    
\item For any formula $RS(Z) \in T$ 
     \begin{enumerate}
     \item $\mathbb{M},RS(X) \vDash^{S} RS(Z)$ iff $R_{-}(Z) \cap R_{-}(X) \not = \emptyset$ and $R^{-}(Z) \cap R^{-}(X) \not = \emptyset$ 
 
     \item $\mathbb{M},RS(X) \vDash^{P} RS(Z)$ iff $R_{-}(Z) \cap R_{-}(X)  = \emptyset$ and $R^{-}(Z) \cap R^{-}(X) \not = \emptyset$ 
     
     \item $\mathbb{M},RS(X) \nvDash^{N} RS(Z)$ iff $R_{-}(Z) \cap R_{-}(X)  = \emptyset$ and $R^{-}(Z) \cap R^{-}(X) = \emptyset$  
     \end{enumerate}
     
\item  Any satisfiability ($\mathbb{M}, RS(X)\vDash^{S}, \mathbb{M}, RS(X)\vDash^{P}$ and $\mathbb{M}, RS(X)\nvDash^{N}$) of $RS(\varPhi) \Delta RS(\varPsi)$, $RS(\varPhi) \nabla RS(\varPsi)$, $RS(\varPhi) \rightarrow RS(\varPsi)$, $RS(\varPhi) \leftarrow RS(\varPsi)$, $RS(\varPhi)^{*}$ and $RS(\varPhi)^{+}$ through the satisfaction relation is validated using (1).       
\end{enumerate}

\begin{remark}
A formula $RS(Z)$ is said to be satisfiable in the model $\mathbb{M}$ with respect to the world $RS(X)$ if $R-TrV(RS(Z))=(1,1)$. It is denoted by $\mathbb{M},RS(X) \vDash^{S} RS(Z)$. A formula $RS(Z)$ is possibly satisfiable if $R-TrV(RS(Z))=(0,1)$ and not satisfiable if $R-TrV(RS(Z))=(0,0)$.
\end{remark}

\begin{remark}
As per the satisfaction relation, if $\mathbb{M},RS(X) \vDash^{S} RS(\{x\})$ then $R^{-}(\{x\}) \subset R^{-}(X)$
\begin{center}
 $\Rightarrow R_{-}(\{x\}) \subset R_{-}(X)$
\end{center}
but $R_{-}(\{x\})=\emptyset$ for any $x \in U$ and $R-TrV(R_{-}(\{x\}))=0$. Therefore, the satisfiability is not feasible for $RS(\{x\})$.
\end{remark}

\subsection{Semantics for rough Heyting-Brouwer predicate Calculus}
Standard semantics is the extension of formal semantics in bi-intuitionistic logic. A hierarchy of standard semantics available within bi-intuitionistic logic systems includes Algebraic semantics, Topological semantics, Beth semantics, Kripke semantics, etc. However, in this section, the standard semantic to be focused on is the Kripke semantic. Kripke semantics offers a method for defining the truth conditions for modal formulations. It interprets the notion of falsity in semantics as a notion of the absence of constructive proof. However, our proposed method dwells on the semantic verification of rough H-B predicate logic through the rough  Kripke semantic. A Kripke model consists of a pre-ordered set that deals with accessibility through the satisfaction relation. An element in the pre-ordered set of a model is a state (or) a world, which decides whether the formulas are satisfied within the state (or) the state that extends it. Using our rough Heyting-Brouwer logic language, the rough Kripke model to be discussed deals with the structure $\mathbb{M}=(RBHey(T), R-TrV)$, where $RBHey(T)$ is a rough bi-Heyting algebra and each $RS(X) \in T$ is considered to be a world in which the formulas are validated. Satisfaction of semantics in the rough Kripke model behaves similarly to the semantics of rough Heyting-Brouwer propositional calculus but the difference in their interpretation is that all the rough bi-Heyting algebraic operations $^{*},^{+},\rightarrow,\leftarrow$ evaluate the formula from the fixed $RS(X) \in T$ to all world that is accessible from $RS(X)$, that is $RS(X) \leq RS(Y)$ (or) $RS(Y) \leq RS(X)$ where $RS(Y) \in T$. This will further broaden the mathematical expression under the satisfaction relation from a fixed world $RS(X) \in T$ to any world that extends it to explore the rough Heyting-Brouwer predicate calculus (RBi-IQC). \\

We begin by defining the structure of the rough Kripke model.

\begin{definition}
A rough Kripke model is a structure $\mathbb{M}=(RBHey(T), R-TrV)$ where
\end{definition}
\begin{enumerate}
	\item $RBHey(T)$ is a rough bi-Heyting algebra 
	\item $R-TrV$ is a truth value function.
\end{enumerate}

\begin{definition}
 For the rough Kripke model $\mathbb{M}=(RBHey(T), R-TrV)$ and the world $RS(X) \in T$, the satisfiability relation is given as follows
\end{definition}
\begin{enumerate}
\item For any formula $RS(Z) \in T$ 
     \begin{enumerate}
     \item $\mathbb{M},RS(X) \vDash^{S} RS(Z)$ iff $R_{-}(Z) \cap R_{-}(X) \not = \emptyset$ and $R^{-}(Z) \cap R^{-}(X) \not = \emptyset$ 
 
     \item $\mathbb{M},RS(X) \vDash^{P} RS(Z)$ iff $R_{-}(Z) \cap R_{-}(X)  = \emptyset$ and $R^{-}(Z) \cap R^{-}(X) \not = \emptyset$ 
     
     \item $\mathbb{M},RS(X) \nvDash^{N} RS(Z)$ iff $R_{-}(Z) \cap R_{-}(X)  = \emptyset$ and $R^{-}(Z) \cap R^{-}(X) = \emptyset$  
     \end{enumerate}
     
\item  The conclusion ($\mathbb{M}, RS(X)\vDash^{S}, \mathbb{M}, RS(X)\vDash^{P}$ and $\mathbb{M}, RS(X)\nvDash^{N}$) of $RS(\varPhi) \Delta RS(\varPsi)$ and $RS(\varPhi) \nabla RS(\varPsi)$ through the satisfaction relation is validated using (1). 

\item \begin{enumerate}
    \item If $\mathbb{M},RS(X) \vDash^{S} RS(\varPhi) \rightarrow RS(\varPsi)$ (or $RS(\varPhi)^{*}$) then $\forall RS(Y)$ such that $RS(X) \leq RS(Y)$
   \begin{enumerate}
       \item $\mathbb{M},RS(Y) \vDash^{S} RS(\varPhi) \rightarrow RS(\varPsi)$ (or $RS(\varPhi)^{*}$)
   \end{enumerate}

   \item If $\mathbb{M},RS(X) \vDash^{P} RS(\varPhi) \rightarrow RS(\varPsi)$ (or $RS(\varPhi)^{*}$) then $\forall RS(Y)$ such that $RS(X) \leq RS(Y)$
   \begin{enumerate}
       \item $\mathbb{M},RS(Y) \vDash^{P} RS(\varPhi) \rightarrow RS(\varPsi)$ (or $RS(\varPhi)^{*}$) (or) $\mathbb{M},RS(Y) \vDash^{S} RS(\varPhi) \rightarrow RS(\varPsi)$ (or $RS(\varPhi)^{*}$)
   \end{enumerate}

   \item If $\mathbb{M},RS(X) \nvDash^{N} RS(\varPhi) \rightarrow RS(\varPsi)$ (or $RS(\varPhi)^{*}$) then $\forall RS(Y)$ such that $RS(X) \leq RS(Y)$
   \begin{enumerate}
       \item $\mathbb{M},RS(Y) \nvDash^{N} RS(\varPhi) \rightarrow RS(\varPsi)$ (or $RS(\varPhi)^{*}$) (or) $\mathbb{M},RS(Y) \vDash^{P} RS(\varPhi) \rightarrow RS(\varPsi)$ (or $RS(\varPhi)^{*}$) (or) (or) $\mathbb{M},RS(Y) \vDash^{S} RS(\varPhi) \rightarrow RS(\varPsi)$ (or $RS(\varPhi)^{*}$)
   \end{enumerate}
\end{enumerate}

\item \begin{enumerate}
    \item If $\mathbb{M},RS(X) \vDash^{S} RS(\varPhi) \leftarrow RS(\varPsi)$ (or $RS(\varPhi)^{+}$) then $\forall RS(Y)$ such that $RS(Y) \leq RS(X)$
   \begin{enumerate}
       \item $\mathbb{M},RS(Y) \vDash^{S} RS(\varPhi) \leftarrow RS(\varPsi)$ (or $RS(\varPhi)^{+}$) (or) (or) $\mathbb{M},RS(Y) \vDash^{P} RS(\varPhi) \leftarrow RS(\varPsi)$ (or $RS(\varPhi)^{+}$) (or) (or) $\mathbb{M},RS(Y) \nvDash^{N} RS(\varPhi) \leftarrow RS(\varPsi)$ (or $RS(\varPhi)^{+}$)
   \end{enumerate}

   \item If $\mathbb{M},RS(X) \vDash^{P} RS(\varPhi) \leftarrow RS(\varPsi)$ (or $RS(\varPhi)^{+}$) then $\forall RS(Y)$ such that $RS(Y) \leq RS(X)$
   \begin{enumerate}
       \item $\mathbb{M},RS(Y) \vDash^{P} RS(\varPhi) \leftarrow RS(\varPsi)$ (or $RS(\varPhi)^{+}$) (or) $\mathbb{M},RS(Y) \nvDash^{N} RS(\varPhi) \leftarrow RS(\varPsi)$ (or $RS(\varPhi)^{+}$)
   \end{enumerate}

   \item If $\mathbb{M},RS(X) \nvDash^{N} RS(\varPhi) \leftarrow RS(\varPsi)$ (or $RS(\varPhi)^{+}$) then $\forall RS(Y)$ such that $RS(Y) \leq RS(X)$
   \begin{enumerate}
       \item $\mathbb{M},RS(Y) \nvDash^{N} RS(\varPhi) \leftarrow RS(\varPsi)$ (or $RS(\varPhi)^{+}$) 
   \end{enumerate}
\end{enumerate}

\end{enumerate}

\textbf{Illustraion 1:} 
 Consider the set of atomic propositions $U=\{x_{1},x_{2},x_{3},x_{4},x_{5},x_{6}\}$ and the equivalence classes are $X_{1}=\{x_{1},x_{3}\}$, $X_{2}= \{x_{2},x_{4},x_{6}\}$ and $X_{3}=\{x_{5}\}$. The model verifies the conclusions derived by the well-formed formulas $RS(\varPhi)$ and $RS(\varPsi)$ through the satisfaction relation of rough Kripke semantics in the world $RS(\{x_{1}\} \cup X_{2} \cup X_{3}) \in T$.\\

The truth values of $X_{1},X_{2}$ and $X_{3}$ for the world $RS(\{x_{1}\} \cup X_{2}  \cup X_{3}) \in T$ are 0, 1 and 1 respectively. The formulas $RS(\varPhi)$ and $RS(\varPsi)$ are defined to be

\begin{center}
 $RS(\varPhi)=(RS(\{x_{2}\} \cup X_{3})^{+} \rightarrow RS(X_{1})) \rightarrow RS(\{x_{1}\} \cup \{x_{2}\})$ \\
 $RS(\varPsi)=(RS(\{x_{1}\} \cup \{x_{2}\}) \leftarrow RS(\{x_{1}\})^{*}) \Delta RS(X_{1} \cup \{x_{2}\})$
\end{center}

\begin{enumerate}
\item  $RS(\varPhi)=(RS(\{x_{2}\} \cup X_{3})^{+} \rightarrow RS(X_{1})) 
       \rightarrow RS(\{x_{1}\} \cup \{x_{2}\})=(RS(X_{1} \cup X_{2}) \rightarrow RS(X_{1})) \rightarrow RS(\{x_{1}\} \cup \{x_{2}\})=RS(X_{1} \cup X_{3}) \rightarrow RS(\{x_{1}\} \cup \{x_{2}\})=RS(\{x_{1}\} \cup X_{2})$ 
       \begin{center}
       $R_{-}(\varPhi) \cap R_{-}(\{x_{1}\} \cup X_{2}  \cup X_{3})=X_{2}$ and $R^{-}(\varPhi) \cap R^{-}(\{x_{1}\} \cup X_{2}  \cup X_{3})=X_{1} \cup X_{2}$\\
       $\Rightarrow \mathbb{M},RS(\{x_{1}\} \cup X_{2} \cup X_{3}) \vDash^{S} 
        RS(\varPhi)$
       \end{center}

\item $RS(\varPsi)=(RS(\{x_{1}\} \cup \{x_{2}\}) \leftarrow RS(\{x_{1}\})^{*}) 
      \Delta RS(X_{1} \cup \{x_{2}\})=(RS(\{x_{1}\} \cup \{x_{2}\}) \leftarrow RS(X_{2} \cup X_{3})) \Delta RS(X_{1} \cup \{x_{2}\})=RS(\{x_{1}\})  \Delta RS(X_{1} \cup \{x_{2}\})=RS(X_{1} \cup \{x_{2}\})$
      \begin{center}
     $R_{-}(\varPsi) \cap R_{-}(\{x_{1}\} \cup X_{2}  \cup X_{3})=\emptyset$ and $R^{-}(\varPsi) \cap R^{-}(\{x_{1}\} \cup X_{2}  \cup X_{3})=X_{1} \cup X_{2} \not = \emptyset$\\
     $\Rightarrow \mathbb{M},RS(\{x_{1}\} \cup X_{2} \cup X_{3}) \vDash^{P} 
        RS(\varPsi)$
      \end{center}

\item $RS(\varPhi) \Delta RS(\varPsi)=RS(\{x_{1}\} \cup X_{2}) \Delta RS(X_{1} \cup 
     \{x_{2}\})=RS(X_{1} \cup X_{2})$
    \begin{center}
    $\Rightarrow \mathbb{M},RS(\{x_{1}\} \cup X_{2} \cup X_{3}) \vDash^{S} 
    RS(\varPhi) \Delta RS(\varPsi)$
    \end{center}

\item $RS(\varPhi) \nabla RS(\varPsi)=RS(\{x_{1}\} \cup X_{2}) \nabla RS(X_{1} \cup 
     \{x_{2}\})=RS(\{x_{1}\} \cup\{x_{2}\})$
    \begin{center}
    $R_{-}(\{x_{1}\} \cup\{x_{2}\}) \cap R_{-}(\{x_{1}\} \cup X_{2} \cup X_{3})= \emptyset$ and $R^{-}(\{x_{1}\} \cup\{x_{2}\}) \cap R^{-}(\{x_{1}\} \cup X_{2} \cup X_{3})=X_{1} \cup X_{2} \not = \emptyset$
    \end{center}
    \begin{center}
    $\Rightarrow \mathbb{M},RS(\{x_{1}\} \cup X_{2} \cup X_{3}) \vDash^{P} 
    RS(\varPhi) \Delta RS(\varPsi)$
    \end{center}

\item $RS(\varPhi) \rightarrow RS(\varPsi)=RS(\{x_{1}\} \cup X_{2}) \rightarrow 
      RS(X_{1} \cup \{x_{2}\})=RS(X_{1} \cup\{x_{2}\})$
       \begin{center}
    $R_{-}(X_{1} \cup\{x_{2}\}) \cap R_{-}(\{x_{1}\} \cup X_{2} \cup X_{3})= \emptyset$ and $R^{-}(X_{1} \cup\{x_{2}\}) \cap R^{-}(\{x_{1}\} \cup X_{2} \cup X_{3})=X_{1} \cup X_{2} \not = \emptyset$
    \end{center}
      
    \begin{center}
    $\Rightarrow \mathbb{M},RS(\{x_{1}\} \cup X_{2} \cup X_{3}) \vDash^{P} 
    RS(\varPhi) \rightarrow RS(\varPsi)$
    \end{center}

     \begin{table}[!ht]
      \small
      \begin{center}
      \label{tab:table1}
      \begin{tabular}{|c|c|c|c|c|} 
      \hline  $\{RS(Y) \mid RS(\{x_{1}\} \cup X_{2} \cup X_{3}) \leq RS(Y)\}$ &  $RS(\varPhi) \rightarrow RS(\varPsi)$ &  $RS(\varPhi)^{*}=RS(\varPsi)^{*}=RS(X_{3})$ & $RS(\varPhi)$ & $RS(\varPsi)$\\ \hline
      $RS(\{x_{1}\} \cup X_{2} \cup X_{3})$ & (0,1) & (1,1) & (1,1) & (0,1) \\ \hline
      $RS(U)$ & (1,1) & (1,1) & (1,1) & (1,1)\\ \hline 
       \end{tabular}
	\caption{Verification for the satisfiability of $RS(\varPhi) \rightarrow 
        RS(\varPsi), RS(\varPhi)^{*}, RS(\varPsi)^{*}$ from the accessible worlds}
         \end{center}
       \end{table}

      From Table 3, the world $RS(U)$ accessible from $RS(\{x_{1}\} \cup X_{2} \cup X_{3})$ provided $RS(\{x_{1}\} \cup X_{2} \cup X_{3}) \leq RS(U)$ then
       \begin{center}
       $\mathbb{M},RS(U) \vDash^{S} RS(\varPhi) \rightarrow RS(\varPsi)$  
       \end{center}
      
\item $RS(\varPhi) \leftarrow RS(\varPsi)=RS(\{x_{1}\} \cup X_{2}) \leftarrow 
      RS(X_{1} \cup \{x_{2}\})=RS(X_{2})$
      \begin{center}
      $\Rightarrow \mathbb{M},RS(\{x_{1}\} \cup X_{2} \cup X_{3}) \vDash^{S} 
      RS(\varPhi) \leftarrow RS(\varPsi)$
      \end{center}

        \begin{table}[!ht]
       \small
       \begin{center}
      \label{tab:table1}
       \begin{tabular}{|c|c|c|c|c|c|} 
      \hline  $\{RS(Y) \mid RS(Y) \leq RS(\{x_{1}\} \cup X_{2} \cup X_{3})\}$ & $RS(\varPhi) \leftarrow RS(\varPsi)$ & $RS(\varPhi)^{+}$ & $RS(\varPsi)^{+}$ & $RS(\varPhi)$ & $RS(\varPsi)$\\ \hline 
       $RS(\{x_{1}\} \cup X_{2} \cup X_{3})$ & (1,1) & (1,1) & (1,1) & (1,1) & (0,1)\\ \hline
       $RS(\{x_{1}\} \cup X_{2})$ & (1,1) & (0,0) & (1,1) & (1,1) & (0,1)\\ \hline
       $RS(X_{2} \cup X_{3})$ & (1,1) & (1,1) & (1,1) & (1,1) & (0,1)\\ \hline 
       $RS(\{x_{2}\} \cup X_{3})$ & (0,0) & (1,1) & (1,1) & (0,0) & (0,0)\\ \hline 
       $RS(\{x_{1}\} \cup \{x_{2}\})$ & (0,0) & (0,0) & (0,0) & (0,0) & (0,0)\\ \hline
       $RS(X_{2})$ & (1,1) & (0,0) & (1,1) & (1,1) & (0,1)\\ \hline
       $RS(X_{3})$ & (0,0) & (1,1) & (1,1) & (0,0) & (0,0)\\ \hline
       $RS(\{x_{2}\})$ & (0,0) & (0,0) & (0,0) & (0,0) & (0,0)\\ \hline
       $RS(\{x_{1}\})$ & (0,0) & (0,0) & (0,0) & (0,0) & (0,0)\\ \hline
       $RS(\emptyset)$ & (0,0) & (0,0) & (0,0) & (0,0) & (0,0)\\ \hline 
        \end{tabular}
	\caption{Verification for the satisfiability of $RS(\varPhi) \leftarrow RS(\varPsi), RS(\varPhi)^{+}, RS(\varPsi)^{+}$ 
        from the accessible worlds}
        \end{center}
        \end{table}
        
        From Table 4, when $RS(Y)=\{RS(X_{2}), RS(X_{2} \cup X_{3}), RS(\{x_{1}\} \cup X_{2}),RS(\{x_{1}\} \cup X_{2} \cup X_{3})\}$, provided $RS(Y) \leq RS(\{x_{1}\} \cup X_{2} \cup X_{3})$
       \begin{center}
       $\exists RS(Y)$: $\mathbb{M},RS(Y) \vDash^{S} RS(\varPhi) \leftarrow RS(\varPsi)$
       \end{center}

\item $RS(\varPhi)=RS(\{x_{1}\} \cup X_{2})$ and $RS(\varPhi)^{*}=RS(X_{3})$\\
      $RS(\varPsi)=RS(X_{1} \cup \{x_{2}\})$ and $RS(\varPsi)^{*}=RS(X_{3})$
      \begin{center}
      $\Rightarrow \mathbb{M},RS(\{x_{1}\} \cup X_{2} \cup X_{3}) \vDash^{S} 
      RS(\varPhi)^{*}$\\
      $\Rightarrow \mathbb{M},RS(\{x_{1}\} \cup X_{2} \cup X_{3}) \vDash^{S} 
      RS(\varPsi)^{*}$
      \end{center}

       From Table 9, for the world $RS(Y)$ provided $RS(\{x_{1}\} \cup X_{2} \cup X_{3}) \leq RS(Y)$ 
       \begin{center}
       $\forall RS(Y)$: $\mathbb{M},RS(Y) \vDash^{S} RS(\varPhi)^{*}$ and  $\mathbb{M},RS(Y) \vDash^{S} RS(\varPsi)^{*}$
       \end{center}

\item $RS(\varPhi)=RS(\{x_{1}\} \cup X_{2})$ and $RS(\varPhi)^{+}=RS(X_{1} \cup 
      X_{3})$\\
      $RS(\varPsi)=RS(X_{1} \cup \{x_{2}\})$ and $RS(\varPsi)^{+}=RS(X_{2} \cup X_{3})$
      \begin{center}
      $\Rightarrow \mathbb{M},RS(\{x_{1}\} \cup X_{2} \cup X_{3}) \vDash^{S} 
      RS(\varPhi)^{+}$\\
      $\Rightarrow \mathbb{M},RS(\{x_{1}\} \cup X_{2} \cup X_{3}) \vDash^{S} 
      RS(\varPsi)^{+}$
      \end{center}     

       From Table 10, for the world $RS(Y)$ and $RS(Y^{\prime})$ provided $RS(Y), RS(Y^{\prime}) \leq RS(\{x_{1}\} \cup X_{2} \cup X_{3})$
       \begin{center}
       $\exists RS(Y)$: $\mathbb{M},RS(Y) \vDash^{S} RS(\varPhi)^{+}$
       \end{center}

       \begin{center}
       $\exists RS(Y^{\prime})$: $\mathbb{M},RS(Y^{\prime}) \vDash^{S} RS(\varPsi)^{+}$
       \end{center}
\end{enumerate}

\textbf{Interpretation}:
In both Illustrations 3 and 4, the satisfiability of the formulas $RS(\varPhi), RS(\varPsi), RS(\varPhi) \Delta RS(\varPsi),$\\$ RS(\varPhi) \nabla RS(\varPsi)$ are verified similar to the semantics of rough H-B propositional calculus. But for formulas obtained through rough bi-Heyting algebraic operations $RS(\varPhi) \rightarrow RS(\varPsi), RS(\varPhi) \leftarrow RS(\varPsi), RS(\varPhi)^{*}$ and $RS(\varPhi)^{+}$, their satisfiability is verified from for the fixed world $RS(X) \in T$ to all the world that extends it. For the intuitionistic negations $\rightarrow$ and $^{*}$, the verification of the satisfiability is extended from $RS(X)$ to all $RS(Y)$, so that $RS(X) \leq RS(Y)$. On the other hand, for the dual intuitionistic negations $\leftarrow$ and $^{+}$, the verification of the satisfiability is extended from $RS(X)$ to all $RS(Y)$, so that $RS(Y) \leq RS(X)$. The semantic satisfiability of the rough Kripke model is different from the formal semantic in a way that formal semantics verifies the satisfiability of formulas in a fixed world but the rough Kripke model semantics verifies the satisfiability of formulas from the fixed world to the accessible worlds.\\

\section{Conclusion}
This study defined rough bi-Heyting algebra in the context of rough semiring $(T,\Delta,\nabla)$, making a novel contribution to the existing literature that primarily deals with the connection between Heyting algebra and rough set. This paper initiated a unique approach in defining rough bi-Heyting algebra on a rough semiring $(T,\Delta,\nabla)$ and contributes to the formalization of rough bi-Heyting algebra in characterizing weaker complements. The established rough bi-Heyting algebraic structure is also employed to develop the semantics of rough Heyting-Brouwer logic. The validation of syntax is verified in the semantics of rough Heyting-Brouwer logic. Looking ahead, future research aims to focus on using this rough bi-Heyting algebraic approach to model diverse semantics and broaden its practical applications.

\section{Acknowledgment}
The authors thank the Management and the Principal, of Sri Sivasubramaniya Nadar College of Engineering for providing the necessary facilities to carry out this work.

\section{Declarations}
The authors declare no conflict of interest.

\end{document}